\documentclass[12pt, reqno]{amsart}
\usepackage{fullpage}
\usepackage[hyphens]{url}

\theoremstyle{plain}
\newtheorem{theorem}{Theorem}[section]
\newtheorem{lemma}[theorem]{Lemma}
\newtheorem{proposition}[theorem]{Proposition}
\newtheorem{corollary}[theorem]{Corollary}
\theoremstyle{remark}
\newtheorem{example}[theorem]{Example}

\title{Commutativity Theorems for Groups and Semigroups}
\author{Francisco Ara\'ujo}
\address[Ara\'ujo]{Col\'{e}gio Planalto, R. Armindo Rodrigues 28, 1600--414 Lisboa, Portugal}
\author{Michael Kinyon$^*$}
\thanks{${}^*$ Partially supported by Simons Foundation Collaboration Grant 359872 and
by FCT project CEMAT-CI\^{E}NCIAS UID/Multi/04621/2013}
\address[Kinyon]{Department of Mathematics, University of Denver, Denver, CO 80208, USA}
\address[Kinyon]{CEMAT-CI\^{E}NCIAS, Departamento de Matem\'{a}tica, Faculdade de Ci\^{e}ncias, Universidade de Lisboa,
1749-016, Lisboa, Portugal}
\email{mkinyon@du.edu}

\begin{document}

\begin{abstract}
In this note we prove a selection of commutativity theorems for various classes of semigroups. For instance, if in a
separative or completely regular semigroup $S$ we have $x^p y^p = y^p x^p$ and $x^q y^q = y^q x^q$ for all $x,y\in S$ where
$p$ and $q$ are relatively prime, then $S$ is commutative. In a separative or inverse semigroup $S$, if there exist three
consecutive integers $i$ such that $(xy)^i = x^i y^i$ for all $x,y\in S$, then $S$ is commutative. Finally, if $S$ is a
separative or inverse semigroup satisfying $(xy)^3=x^3y^3$ for all $x,y\in S$, and if the cubing map $x\mapsto x^3$ is
injective, then $S$ is commutative.
\end{abstract}

\maketitle

\section{Introduction}

Broadly speaking, a \emph{commutativity theorem} in group theory is any result concluding that a group is commutative, \emph{i.e.} abelian.
Perhaps the best known example is the following standard exercise, usually given to students at the beginning of their study of group theory:
\begin{center}
\emph{If $G$ is a group satisfying $x^2 = 1$ for all $x\in G$, then $G$ is commutative.}
\end{center}

Commutativity theorems can sometimes be extended to various classes of semigroups properly containing groups. For instance, a semigroup $S$ is \emph{cancellative} if it
satisfies the conditions $xy=xz\implies y=z$ and $yx=zx\implies y=z$ for all $x,y,z$. Every finite cancellative semigroup is a group; the positive integers under
addition provide an example of a cancellative semigroup which is not a group.

The exercise above extends easily to cancellative semigroups once we reinterpret the condition ``$x^2 = 1$''. Since we do not wish to assume the existence of an
identity element, we replace the condition with ``$x^3 = x$'', which is clearly equivalent to $x^2 = 1$ in groups.

\begin{proposition}
\label{Prp:easy}
  If $S$ is a cancellative semigroup satisfying $x^3 = x$ for all $x\in S$, then $S$ is commutative (and in fact, is a group satisfying $x^2 = 1$ for all $x\in S$).
\end{proposition}

The parenthetical part of the assertion suggests the proof: if $x^3 = x$, then $x^3 y = xy$ and so cancelling gives $x^2 y = y$ for all $x,y\in S$.
Dually, $y x^2 = y$ for all $x,y\in S$. Thus for all $x,y\in S$, $x^2 = x^2 y^2 = y^2$. This constant, which we denote by $1$, is an identity element.
Hence $x^3 = x = x1$, and cancelling gives $x^2 = 1$ for all $x\in S$. Therefore $S$ is a group and we have reduced the problem to the original exercise.

Still using this elementary example to illustrate our point, further extensions of the result are possible. A semigroup is \emph{separative} if it satisfies the
conditions $xy=xx\ \&\ yx=yy\implies x=y$ and $xy=yy\ \&\ yx=xx\implies x=y$ (\cite{Petrich}, Def. II.6.2, p. 51). Every cancellative semigroup is evidently
separative.

We also need the notion of a semilattice of semigroups. A semilattice is a partially ordered set $(I,\leq)$ such that every two elements $x,y\in S$
have a greatest lower bound, denoted by $x\land y$. A semigroup $S$ is a \emph{semilattice of semigroups} if there exist a semilattice  $(I,\leq)$ and a set
$Y=\{S_\alpha \}_{\alpha\in I}$ of pairwise disjoint subsemigroups $S_{\alpha}\leq S$ indexed by $I$ such that $S=\cup_{\alpha \in I} S_\alpha$,
and satisfying this property: for all $\alpha,\beta \in I$ and for all $a\in S_\alpha$, $b\in S_\beta$, we have $ab\in S_{\alpha\land \beta}$.
(For details, see \cite{Petrich}, Def. II.1.4, p. 27).

For our purposes, the following results are key (\cite[Thm. II.6.4, p. 51]{Petrich}, \cite[Thm. 3.12, p. 47]{Nagy}).

\begin{proposition}
\label{Prp:semilattice}
Let $S$ be a semigroup.
\begin{enumerate}
  \item $S$ is separative if and only if $S$ is semilattice of cancellative semigroups.
  \item $S$ is commmutative and separative if and only if $S$ is a semilattice of commutative cancellative semigroups.
\end{enumerate}
\end{proposition}

Now we can extend the original exercise even further.

\begin{corollary}
  Let $S$ be a separative semigroup satisfying $x^3 = x$ for all $x\in S$. Then $S$ is commutative, and in fact, is a semilattice of abelian groups
  satisfying $x^2 = 1$.
\end{corollary}

Indeed, by Proposition \ref{Prp:semilattice}(1), $S$ is a semilattice of cancellative semigroups $S_{\alpha}$, $\alpha\in I$,
such that each $S_{\alpha}$ satisfies $x^3 = x$. By Proposition \ref{Prp:easy}, each $S_{\alpha}$ is an abelian group satisfying
$x^2 = 1$. By Proposition \ref{Prp:semilattice}(2), $S$ is commutative.

\smallskip

We can also view the generalizations of our exercise from a different perspective. A semigroup $S$ is \emph{regular} if for each $a\in S$,
there exists $b\in S$ such that $aba = a$. This is equivalent to asserting that each $a\in S$ has an \emph{inverse} $a'\in S$ satisfying
$aa'a = a$ and $a'aa'=a'$. If each $a\in S$ has a \emph{unique} inverse, then $S$ is said to be an \emph{inverse semigroup}. Equivalently,
an inverse semigroup is precisely a regular semigroup in which all the idempotents commute. A semigroup $S$ such that each element has a
commuting inverse $aa'=a'a$ is said to be \emph{completely regular}. Equivalently, a completely regular semigroup is a union of groups.
A completely regular, inverse semigroup is called a \emph{Clifford semigroup}. A Clifford semigroup is characterized as a semilattice of
groups. (For further details on regular semigroups, see, for instance, \cite{Howie}.) Note in particular that \emph{every Clifford semigroup is separative}.

This gives us a different way of viewing our exercise.
A semigroup satisfying $x^3 = x$ for all $x$ is completely regular, with the commuting inverse of each $x$ being given by $x' = x$.
It is easy to see that a regular, cancellative semigroup is a group by essentially the same argument as above: $xx'xy = xy$ and $yxx'x = yx$,
so $x'xy=y$ and $yxx'=y$ for all $x,y\in S$. A semigroup (or more generally, any magma) with both a left identity element and a right identity
element has a (necessarily unique) identity element $1$, and we have $xx' = x'x = 1$ for all $x\in S$. Thus we have another proof of
Proposition \ref{Prp:semilattice}.

\medskip

In this paper, we will extend three commutativity theorems from group theory to semigroups. Our first result, which was our original motivation,
is based on a recent preprint of Venkataraman \cite{Venka}. She proved that in a finite groups, if squares commute with squares and cubes commute with cubes
then the group is commutative; she also proposed the problem of extending her result to infinite groups. More generally, in the same paper, she asked if a group satisfying the conditions $x^p y^p = y^p x^p$ and $x^q y^q = y^q x^q$ for all $x$ where $p$ and $q$ are
relatively prime is necessarily commutative. Although we did not know it when we began our investigation, this is apparently a folk result in group theory \cite{exchange}, although we have not been able to find a reference in the literature. (Note that the proofs given in the cited website
do not generalize directly to semigroups.) In the spirit of our discussion above, we prove
that Venkataraman's desired result holds more generally.

\begin{theorem}
\label{Thm:Main1}
Let $S$ be a separative or completely regular semigroup such that, for all $x,y\in S$, $x^p y^p = y^p x^p$ and $x^q y^q = y^q x^q$
where $p$ and $q$ are relatively prime positive integers. Then $S$ is commutative.
\end{theorem}

In Example \ref{Ex:not_inverse}, we note that this theorem cannot be extended from completely regular
semigroups to general regular semigroups.

\smallskip

It is easy to see that a group, or more generally a cancellative semigroup, is commutative if and only if $(xy)^2 = x^2 y^2$ for all $x,y$.
The direct implication is trivial; for the converse, $xyxy=xxyy$ implies $yx=xy$ after cancellation. In the same vein is the following
well-known exercise (\cite{Herstein}, {\S}2.3, Exer. 4):

\begin{quote}
  If $G$ is a group such that $(ab)^i =a^i b^i$ for three consecutive integers $i$ for all $a,b\in G$, show that $G$ is abelian.
\end{quote}

The slightly awkward wording allows two interpretations: that the integers $i$ depend on the elements $a, b$, or that the same
integers $i$ work for all $a,b$. The proof for groups is essentially the same in either case. Our generalizations require both readings.

\begin{theorem}
\label{Thm:Main2}
Let $S$ be a semigroup.
\begin{enumerate}
  \item Suppose $S$ is separative and suppose that for each $a,b\in S$, there exist three consecutive nonnegative integers $i$ such that
  $(ab)^i = a^i b^i$. Then $S$ is commutative.
  \item Suppose $S$ is an inverse semigroup and suppose that there exist three consecutive nonnegative integers $i$ such that
   $(ab)^i = a^i b^i$ for all $a,b\in S$. Then $S$ is commutative.
\end{enumerate}
\end{theorem}

Part (2) of this theorem cannot be formulated in the same way as part (1) is; see Example \ref{Ex:reg_bad}.
In addition, part (2) cannot be generalized to other types of regular semigroups; see Example
\ref{Ex:not_reg}.

Another commutativity theorem for groups was motivated for us by another known exercise \cite[Exer. 24, p. 48]{Herstein}:
\begin{quote}
  Let $G$ be a finite group whose order is not divisible by $3$. Suppose
  that $(ab)^3 = a^3 b^3$ for all $a, b\in G$. Prove that G must be abelian.
\end{quote}
The finiteness is not essential and the condition can be replaced with the assumption that $G$ is a group
with no elements of order $3$, that is, $G$ satisfies the condition
\begin{equation}\label{no_order_3}
  x^3 = 1\implies x = 1
\end{equation}
for all $x\in G$.
More generally, a group $G$ satisfying $(ab)^3 = a^3 b^3$ for all $a,b\in G$ can be described by a more general theorem
of Alperin \cite{Alperin} as a quotient of a subgroup of a direct product of abelian groups and
groups of exponent $3$. The condition \eqref{no_order_3} rules out groups of exponent $3$, and so $G$ is abelian.
The existence of nonabelian groups of exponent $3$, such as the unique one
of order $27$, shows that some additional hypothesis like \eqref{no_order_3} is needed to conclude commutativity.

There are two reasonable reformulations of \eqref{no_order_3} for semigroups. First, if a semigroup $S$ satisfies
$(ab)^3 = a^3 b^3$ for all $a,b\in S$, then this just asserts that the cubing mapping $S\to S;x\mapsto x^3$ is
an endomorphism. From this point of view, \eqref{no_order_3} asserts that in groups, the kernel of this endomorphism
is trivial, or equivalently, that the endomorphism is injective. This latter formulation makes sense in
any semigroup $S$:
\begin{equation}\label{cube_inj}
  x^3 = y^3 \implies x = y
\end{equation}
for all $x\in S$.

Another reformulation of \eqref{no_order_3} for groups which works for any semigroup $S$,  possibly without an identity element,
is weaker, but more straightforward:
\begin{equation}\label{4_to_2}
  x^4 = x\implies x^2 = x
\end{equation}
for all $x\in S$. To see that this is weaker, suppose \eqref{cube_inj} holds and $x^4 = x$. Then
$(x^2)^3 = x^6 = x^3$, and so applying \eqref{cube_inj} yields $x^2 = x$.

\begin{theorem}	
\label{Thm:Main3}
Let $S$ be a semigroup satisfying $(xy)^3=x^3y^3$ for all $x\in S$.
\begin{enumerate}
  \item If $S$ is separative and satisfies \eqref{cube_inj}, then $S$ is commutative.
  \item If $S$ is an inverse semigroup and satisfies \eqref{4_to_2}, then $S$ is commutative.
\end{enumerate}
\end{theorem}

The hypothesis of part (1) of Theorem~\ref{Thm:Main3} cannot be weakened to \eqref{4_to_2}; see
Example \ref{Ex:bad_cancel} below. Also, neither part of the theorem extends to other types of
regular semigroups; see Example \ref{Ex:reg_bad2}.

\medskip

All of our investigations were aided by the automated deduction tool \texttt{Prover9} created by McCune \cite{McCune}. Automated theorem provers are especially good at equational reasoning, being able to derive
consequences of equational axioms much faster and more efficiently than humans.
Any currently available automated theorem prover would have sufficed for this project, but \texttt{Prover9}
has the advantage that its input and output are easily readable by mathematicians with no familiarity with
such tools. For example, here are the axioms in an input file for the special case of
Theorem \ref{Thm:Main1} where $S$ is a group, $p = 2$ and $q = 3$:
\begin{verbatim}
% group axioms
(x * y) * z = x * (y * z). % associativity
e * x = x. x * e = x.      % identity element
x' * x = e. x * x' = e.    % inverses

% x^2 y^2 = y^2 x^2
(x * x) * (y * y) = (y * y) * (x * x).

% x^3 y^3 = y^3 x^3
(x * (x * x)) * (y * (y * y)) = (y * (y * y)) * (x * (x * x)).
\end{verbatim}
The goal is just
\begin{verbatim}
x * y = y * x. % commutativity
\end{verbatim}
Here each equation is interpreted by \texttt{Prover9} to be universally quantified in the variables.
Everything written after a \% symbol is a comment. Notice that the association of terms in any equation
is made explicit; while there are settings in \texttt{Prover9} which allow one
to avoid parenthesization of the input, they do not generally improve readability of proofs.

Since we are working in semigroups, the associative law is part of any input file, and is heavily used
throughout proofs, mostly as a rewrite rule. While this can lengthen proofs, it actually causes little
to no trouble for a human reader, who can skip many lines where the rule is being applied.

Running \texttt{Prover9} with its default settings on the above input gives a proof in less than half
a minute on a not particularly fast computer. The proof has 174 steps. Some of the steps are long
and unpleasant; for instance, the longest one has an equation with 67 symbols in it (variables and
occurrences of the operations $\ast$ and ${}'$. Here is a more typical one from the proof, split into
two lines: 
\begin{verbatim}
961 x * (y * (x' * (y * y))) = y * (y * (x * (y * x'))).
   [para(864(a,2),51(a,1,2,2)),rewrite([14(7)])].
\end{verbatim}
The number 961 is a clause identification number which is an internal index that \texttt{Prover9}
uses to keep track of kept clauses. The part in square brackets is the justification for 
clause 961. Here ``para'' is short for \emph{paramodulation}, which is the primary inference
used in equational reasoning. Paramodulation refers to the substitution of one side of
an equation into a subterm of another equation. In this case, clause 864 was plugged into a
subterm of clause 51. This was followed by a rewrite of the resulting clause by clause 14.

It is not terribly enlightening to show the details of this particular step in the proof nor
any other step, because it turns out not to be necessary for translation into humanly readable
form. A proof of 167 steps is, by the standards of automated theorem provers, not very long, and
so it is reasonable to try to obtain a human proof. For familiar associative structures such 
as groups, lattices, rings, and so on, this is usually easy, albeit sometimes time-consuming.
Given two equations and being told that under the axioms of group theory, the two yield a third
is usually enough for someone familiar with groups to see how the proof goes. 

In addition, a human reader can take numerous shortcuts. For example, here is another step in
the proof, omitting the justification:
\begin{verbatim}
11605 x * (y * (x * y')) = x * x.
\end{verbatim}
A human reader can immediately see that the proof is essentially finished: cancel $x$ on
each side of the identity and then multiply on the right by $y$ to get commutativity. 
\texttt{Prover9}, on the other hand, took 14 additional steps to reach commutativity. 
In other words, the ``out of the box'' proof that \texttt{Prover9} found was far from 
optimal. 

Experienced users can tweak \texttt{Prover9}'s many parameters and use various specialized
techniques to find proofs faster, to find shorter proofs, and so on. For example, changing
the term ordering from the default lexicographic path ordering (LPO) to the Knuth-Bendix
ordering gets a different proof in just 10 seconds and the new proof is 19 steps shorter than
the first one. 

Certainly the most interesting use of \texttt{Prover9} was in our investigation of Theorem \ref{Thm:Main1}.
Conditions such as $x^p y^p = y^p x^p$ for all $x,y$ where $p$ is an arbitrary but fixed positive integer
cannot be
directly encoded in \texttt{Prover9} (or any other first-order theorem prover) because it has no built-in
description of the integers. Instead, we had to look at several special cases such as the one above.

It was only after examination of several special cases that we realized that many of the steps in the
proofs were similar. This enabled us to see the pattern of the proof of Theorem \ref{Thm:Main1} for
groups. (Recall that at this point, we were not aware that the theorem was a known
folk result.) In particular, the special cases led us to our formulation of Lemma \ref{g_lemma} below
as containing the essential idea of the proofs. It was also at this same point in our investigations
that we realized the theorem holds more generally in cancellative, and then separative semigroups.
This motivated us to
look at a sampling of other commutativity theorems in various classes of semigroups.

After this paper was submitted, we became aware of the recently published paper of 
Moghaddam and Padmanabhan \cite{MP}. The results contained therein are different from ours, but
the spirit of the work is exactly the same: by extracting the essential features of syntactic
proofs of commutativity theorems for groups, they were able to extend them to cancellative semigroups.

\section{Proof of Theorem \ref{Thm:Main1}}

The goal of this section is to prove Theorem \ref{Thm:Main1}.
We start with the following key lemma.

\begin{lemma}\label{g_lemma}
Let $S$ be a cancellative semigroup and suppose that
 there exists a map $g:S\to S$ satisfying the following conditions: for all $x,y\in S$,
\begin{enumerate}	
\item[(a)] \label{ax1} $x g(x) = g(x) x$;
\item[(b)] \label{ax2} $g(x) g(y) = g(y) g(x)$;
\item[(c)] \label{ax3} $x g(x)\cdot y g(y) = y g(y)\cdot x g(x)$.
\end{enumerate}
Then the semigroup $S$ is commutative.
\end{lemma}
\begin{proof}
We claim that the following identity holds:
\begin{equation}\label{ax4}
  g(g(x) y) y = y g(g(x) y)\,.
\end{equation}

In fact,
\begin{align*}
\underbrace{g(x) y}_{u} \underbrace{g(g(x) y)}_{g(u)}
&\stackrel{(a)}{=} \underbrace{g(g(x) y)}_{g(u)} \underbrace{g(x) y}_u \\
&= \underbrace{g(g(x) y)}_{g(u)} \underbrace{g(x)}_{g(x)} y \\
&\stackrel{(b)}{=} \underbrace{g(x)}_{g(x)} \underbrace{g(g(x) y)}_{g(u)}  y,
\end{align*}
which, eliminating $g(x)$ by left cancellation, gives (\ref{ax4}).

Our next claim is that $g(x)y=yg(x)$. We start by observing that
\begin{equation}\label{ax5}
y g(y) \underbrace{g(x) y}_{u} \underbrace{g(g(x) y)}_{g(u)} \stackrel{(c)}{=} \underbrace{g(x) y}_{u} \underbrace{g(g(x) y)}_{g(u)} y g(y).
\end{equation}
Now,
\begin{align*}
  g(x) y \underbrace{g(y) g(g(x) y) }_{g(y)g(u)} y & \stackrel{(b)}{=} g(x) y \underbrace{g(g(x) y)g(y)}_{g(u)g(y)}y \\
  &= g(x) y g(g(x) y)\underbrace{g(y)y} \\
  &\stackrel{(a)}{=} g(x) y g(g(x) y) \underbrace{y g(y)} \\
  &\stackrel{(\ref{ax5})}{=} yg(y) g(x) y g(g(x) y)  \\
  &\stackrel{(b)}{=} y \underbrace{g(x) g(y)} yg(g(x) y) \\
  &\stackrel{(\ref{ax4})}{=} y g(x) g(y) \underbrace{g(g(x) y)y},
\end{align*}
yielding
\[
g(x) y\cdot  \Big( g(y) g(g(x) y) y\Big) = y g(x)\cdot \Big( g(y) g(g(x) y) y\Big),
\]
which, by right cancellation, implies $g(x)y=yg(x)$, as claimed.

Now the proof that $xy=yx$ is straightforward:
\[
x \underbrace{y g(x)} g(y) = x \underbrace{g(x) y} g(y) \stackrel{(c)}{=} y {g(y) x} g(x)=
y \underbrace{x g(y)} g(x) \stackrel{(b)}{=} y x g(x) g(y),
\]
and $xy=yx$ follows by right cancellation. This completes the proof of the lemma.
\end{proof}

The previous lemma  opens the gate to the proof of our first main theorem.

\begin{proof}[Proof of Theorem \ref{Thm:Main1}]
We may assume without loss of generality that $p, q > 1$.

Suppose first that $S$ is cancellative. Since $p$ and $q$ are relatively prime, by Bezout's identity there exist integers
$r,s$ such that $pr + qs = 1$. Since one of $pr$ or $qs$ must be negative, we assume without loss of generality that $qs < 0$;
thus $-qs > 0$ and $x^{-qs}\in S$ for all $x\in S$. Since $q>0$, we have $s<0$ so that $-s>0$; thus $x^{-s}\in S$ for all $x\in S$.
As $pr>0$ and $p>0$, we have $r>0$ and $x^{pr}, x^r\in S$ for all $x\in S$.

Let $g(x) = x^{-qs}$. We claim that $g(x)$ satisfies the three properties (a), (b) and (c)  of the previous lemma. By associativity,
we have $xg(x)=g(x)x$ so that (a) holds. Regarding (b) we have
\[
g(x)g(y) = x^{-qs}y^{-qs} = (x^{-s})^q (y^{-s})^q 
= (y^{-s})^q (x^{-s})^q = y^{-qs}x^{-qs} = g(y)g(x).
\]
The second equality holds because by assumption the $q$th powers commute.
Finally, regarding (c) we have
\begin{alignat*}{4}
xg(x)yg(y) &= xx^{-qs}yy^{-qs} &&= x^{1-qs}y^{1-qs} &&= x^{pr}y^{pr}\\
           &= (x^r)^p (y^r)^p  &&= (y^r)^p (x^r)^p  &&= y^{pr}x^{pr}\\
           &= y^{1-qs}x^{1-qs} &&= yy^{-qs}xx^{-qs} &&= yg(y)xg(x).
\end{alignat*}
The fourth equality holds because $p$th powers commute.

We have proved that $S$ admits a function $g:S\to S$ satisfying the three conditions of the
previous lemma. It follows that $S$ is commutative.

Next assume that $S$ is separative. By Proposition \ref{Prp:semilattice}, $S$ is a semilattice of cancellative
semigroups $S_{\alpha}$, each of which satisfies the hypotheses of the theorem. It follows that each $S_{\alpha}$
is commutative. By Proposition \ref{Prp:semilattice} again, $S$ is commutative.

Finally, assume that $S$ is completely regular. If $e,f\in S$ are idempotents, then
$ef = e^p f^p = f^p e^p = fe$. It follows that $S$ is an inverse semigroup. Since $S$ is both completely regular
and inverse, it is a Clifford semigroup, hence is a semilattice of groups. In particular, $S$ is separative and
the desired result follows.
\end{proof}

\begin{example}\label{Ex:not_inverse}
  Theorem \ref{Thm:Main1} does not generalize from completely regular semigroups to other types of
  regular semigroups. For example, let $S$ be the Brandt semigroup of order $5$. Then for every positive
  integer $p$ and every $x\in S$, $x^p$ is an idempotent. Thus the hypotheses of the theorem are satisfied
  since idempotents commute in inverse semigroups, but $S$ is not commutative.
\end{example}

\section{Proof of Theorems \ref{Thm:Main2}}

We first need a lemma which will prove useful in both this section and the next.

\begin{lemma}\label{Lem:powers}
  Let $S$ be an inverse semigroup and suppose there exists an integer $k > 1$ such that
  $(xy)^k = x^k y^k$ for all $x\in S$. Then $S$ is a Clifford semigroup.
\end{lemma}
\begin{proof}
  Denote the unique inverse of an element $x\in S$ by $x'$. We will show that $xx' = x'x$
  for all $x\in S$. It will follow that $S$ is completely regular, hence Clifford.
 First, since $(x'x)^k = x'x$, we have
 \begin{equation}\label{tmp1}
   (x')^k x^k = x'x
 \end{equation}
 for all $x\in S$. Next, recalling that $(x')^{k-1} = (x^{k-1})'$ in inverse semigroups, we compute
 \[
 (x')^{k-1} x^k = (x^{k-1})' x^k = \underbrace{(x')^{k-1} x^{k-1}\cdot xx'}x
 = xx'\cdot (x')^{k-1} x^{k-1} x = x(x')^k x^k
 \stackrel{\eqref{tmp1}}{=} xx'x = x\,,
 \]
 where we used the fact that idempotents commute in the third equality. Thus
 \begin{equation}\label{tmp2}
   (x')^{k-1} x^k = x\,.
 \end{equation}
 Next, we have
\[
x'xx \stackrel{\eqref{tmp2}}{=} x'x(x')^{k-1}x^k = (x')^{k-1} x^k  \stackrel{\eqref{tmp2}}{=} x\,,
\]
where we used $k > 1$ in the second equality. Thus we have both
\begin{align}
  x'xx &= x\,, \label{tmp3a} \\
  xx'x' &= x'\,, \label{tmp3b}
\end{align}
where \eqref{tmp3b} follows from \eqref{tmp3a} by replacing $x$ with $x'$ and using $x'' = x$.
Finally, we compute
\[
xx' \stackrel{\eqref{tmp3a}}{=} x'xxx' = xx'x'x \stackrel{\eqref{tmp3b}}{=} x'x\,,
\]
where we used commuting idempotents in the second equality. This completes the proof of the lemma.
\end{proof}

\begin{proof}[Proof of Theorem \ref{Thm:Main2}]
We begin with part (1) and suppose first that $S$ is cancellative. One of the
standard proofs of Herstein's exercise (\cite{Herstein}, {\S}2.3, Exer. 4) only uses cancellation and so
applies here. Say that the three consecutive nonnegative integers are $i, i+1, i+2$. Then
$a^i b^i ab = (ab)^i ab = (ab)^{i+1} = a^{i+1} b^{i+1}$. Cancel $a^i$ on the left and $b$ on the right to get
$b^i a = a b^i$. Repeating the same argument with $i+1$ in place of $i$ gives $b^{i+1} a = a b^{i+1}$.
Thus $b^i\cdot ab = a b^i b = ab^{i+1} = b^{i+1} a = b^i\cdot ba$. Cancelling gives $ab = ba$.

Now suppose $S$ is separative. By Proposition \ref{Prp:semilattice}, $S$ is a semilattice of cancellative semigroups $S_{\alpha}$.
Each $S_{\alpha}$ satisfies the hypothesis of the theorem, hence is commutative. By Proposition \ref{Prp:semilattice},
$S$ is commutative.

Now we turn to part (2) and assume that $S$ is an inverse semigroup satisfying the hypotheses of the theorem.
Applying Lemma \ref{Lem:powers} with $k = i+1$, we have that $S$ is a Clifford semigroup. In particular, $S$ is separative. Since
the hypotheses of part (2) are stronger than those of part (1), we may now apply part (1) to conclude that $S$
is commutative.
\end{proof}

\begin{example}\label{Ex:reg_bad}
  Let $S$ be any nonClifford inverse semigroup with a zero $0$, for instance, the Brandt semigroup
  of order $5$ will suffice.
  Let $b$ be any element such that $bb' \neq b'b$. Then $(0b)^k = 0^k b^k$ for all $k\geq 1$. This shows that part (2) of Theorem \ref{Thm:Main2} cannot be strengthened to be like part (1).
\end{example}

\begin{example}\label{Ex:not_reg}
Let $S = \{e,f\}$ be the 2-element left (say) zero semigroup. Then trivially $(xy)^k = x^k y^k$ for
all $x,y\in S$ and all positive integers $k$, but $S$ is not commutative. Thus Theorem \ref{Thm:Main2}
does not extend to arbitrary regular semigroups or even completely regular semigroups.
\end{example}

\section{Proof of Theorem \ref{Thm:Main3}}

We start with a lemma of some independent interest.

\begin{lemma}\label{Lem:cubes_commute}
  Let $S$ be a cancellative semigroup satisfying $(xy)^3 = x^3 y^3$ for all $x,y\in S$. Then for all $x,y\in S$,
  \begin{equation}\label{Eq:cubes_commute}
    x^3 y = y x^3\,.
  \end{equation}
\end{lemma}
\begin{proof}
  First, cancellation on both sides of $(xy)^3 = x^3 y^3$ gives $(yx)^2 = x^2 y^2$ for all $x,y\in S$.
  Using this, we compute $(x\cdot yx)(x\cdot yx) = (yx)^2 x^2 = x^2 y^2 x^2$. Cancelling on both sides, we obtain
  \begin{equation}\label{engel}
    xy^2 x = yx^2 x
  \end{equation}
  for all $x,y\in S$. Next, $x\cdot y^2 x^2\cdot x = x(xy)^2 x \stackrel{\eqref{engel}}{=} xy\cdot x^2\cdot xy$.
  Cancelling $xy$ on the left gives $yx^3 = x^3 y$ for all $x\in S$, as desired.
\end{proof}

Finally, we prove our last main result.

\begin{proof}[Proof of Theorem \ref{Thm:Main3}]
  For part (1), assume first that $S$ is cancellative and satisfies \eqref{cube_inj}. Then \eqref{Eq:cubes_commute}
  shows that the image $C = \{x^3\mid x\in S\}$ of the cubing map $S\to S;x\mapsto x^3$ is commutative.
  The condition \eqref{cube_inj} asserts that this map is injective, hence $S$ is isomorphic to $C$. In particular,
  $S$ is commutative.

  Now assume $S$ is separative. By Proposition \ref{Prp:semilattice}, $S$ is a semilattice of cancellative semigroups $S_{\alpha}$, each
  of which satisfies both \eqref{cube_inj} and $(xy)^3 = x^3 y^3$ for all $x\in S_{\alpha}$. By the argument above, each $S_{\alpha}$
  is commutative. Applying Proposition \ref{Prp:semilattice} again, we have that $S$ is commutative.

  For (2), now let $S$ be an inverse semigroup satisfying both \eqref{4_to_2} and $(xy)^3 = x^3 y^3$ for all $x\in S$.
  By Lemma \ref{Lem:powers}, $S$ is a Clifford semigroup, hence a semilattice of groups $S_{\alpha}$. Each group $S_{\alpha}$ satisfies \eqref{4_to_2}
  as well, but in groups, \eqref{4_to_2} is equivalent to \eqref{cube_inj}. In particular, $S$ is a separative semigroup
  satisfying the conditions of part (1), and so $S$ is commutative.
\end{proof}

\begin{example}\label{Ex:bad_cancel}
 The hypothesis of part (1) of Theorem \ref{Thm:Main3} cannot be weakened to \eqref{4_to_2}. Indeed, let
$S = \left\{ \begin{bmatrix}
                     1 & a \\
                     0 & b
\end{bmatrix} \mid a,b \in \mathbb{Z}^+ \right\}$ with matrix multiplication as the operation.
Then $S$ is a cancellative semigroup without idempotents and thus trivially satisfies \eqref{4_to_2}. However,
$S$ is not commutative.
\end{example}

\begin{example}\label{Ex:reg_bad2}
  Let $S$ be as in Example \ref{Ex:reg_bad} and note once again that $(xy)^3 = x^3 y^3$ is trivially
  satisfied for all $x,y\in S$. Since $S$ is idempotent, conditions \eqref{cube_inj} and \eqref{4_to_2}
  both hold. However $S$
  is not commutative. Thus neither part of Theorem \ref{Thm:Main3} extends to other types of regular
  semigroups. 
\end{example}

\end{document}